\documentclass{article}
\usepackage{algorithm}
\usepackage{overpic}
\usepackage{eepic}
\usepackage{subfig}
\usepackage{tikz}
\usepackage{graphicx}
\usepackage{algorithmic}
\usepackage{array}
\usepackage{bm}
\usepackage{amsfonts}
\usepackage{amsmath}
\usepackage{amssymb}
\usepackage{mathrsfs}

\DeclareMathOperator*{\argmin}{arg\,min}
\newtheorem{thm}{Theorem}[section]
\newtheorem{remark}[thm]{Remark}

\numberwithin{equation}{section}
\newtheorem{lemma}[thm]{Lemma}
\newtheorem{definition}[thm]{Definition}
\newenvironment{proof}{{\bf Proof.}}{\hfill$\square$\vskip.5cm}

\title{Rank Approximation of a Tensor with Applications in Color Image and Video Processing}

\author{Ramin Goudarzi\footnote{Department of Mathematics, University of Alabama at Birmingham, Birmingham, AL 35294, USA, ramin@uab.edu.}  \   Carmeliza Navasca\footnote{Department of Mathematics, University of Alabama at Birmingham, 1300 University Boulevard, Birmingham, AL, 35294, USA, cnavasca@uab.edu} \ Da Yan\footnote{Department of Computer Science, University of Alabama at Birmingham, Birmingham, AL 35294, USA, yanda@uab.edu.}}

\date{\today}

\begin{document}

\maketitle
 
\begin{abstract}
We propose a block coordinate descent type algorithm for estimating the rank of a given tensor. In addition, the algorithm provides the canonical polyadic decomposition of a tensor. In order to estimate the tensor rank we use sparse optimization method using $\ell_1$ norm. The algorithm is implemented on single moving object videos and color images for approximating the rank.
\setcounter{section}{0}


\vspace{8pt}
\noindent
\end{abstract}

\section{Introduction}

In 1927, Hitchcock \cite{Hitch1,Hitch2} proposed the idea of the polyadic form of a tensor, i.e., expressing a tensor, multilinear array, as the sum of a finite number of rank-one tensors. This decomposition is called the canonical polyadic (CP) decompositon; it is known as CANDECOMP or PARAFAC. It has been extensively applied to many problems in various engineering \cite{Sid1,Sid2,Acar,DeVos} and science \cite{Smilde,Kroonenberg}. Specifically, tensor methods have been applied in many multidimensional datasets in signal processing applications \cite{Comon1,Comon3,Lieven}, color image processing \cite{colorpaper1,colorpaper2} and video processing \cite{bouwmans2,bouwmans1}. Most of these applications rely on decomposing a tensor data into its low rank form to be able to perform efficient computing and to reduce memory requirements. In computer vision, detection of moving objects in video processing relies on foreground and background separation, i.e. the separation of the moving objects called foreground from the static information called background, requires low rank representation of video tensor. In color image processing, the rgb channels in color image representation requires extensions of the matrix models of gray-scale images to low rank tensor methods.
There are several numerical techniques \cite{comon,domanov,kolda,b1,Sid1} for approximating a low rank tensor into its CP decomposition, but they do not give an approximation of the minimum rank. In fact, most low rank tensor algorithms require an a priori tensor rank to find the tensor decomposition. Several theoretical results
\cite{Kruskal1,Landsberg} on tensor rank can help, but they are limited to low-multidimensional and low order tensors.

In this work, the focus is on finding an estimation of the tensor rank and its rank-one tensor decomposition (CP) of a given tensor. There are also algorithms \cite{Comon1,brachat} which give tensor rank, but they are specific to symmetric tensor decomposition over the complex field using algebraic geometry tools. Our proposed algorithm addresses two difficult problems for the CP decomposition: (a) one is that finding the rank of tensors is a NP-hard problem \cite{hillar} and (b) the other is that tensors can be ill-posed \cite{silva} and failed to have their best low-rank approximations. 
 
The problem of finding the rank of a tensor can be formulated as a constrained optimization problem. 
\begin{equation*}
\min_\alpha \Vert \alpha \Vert_0 \, \, \, 
\text{s.t.} \, \, \, \,  \mathcal{X} = \sum_{r=1}^R \alpha_r ( a_r \circ b_r \circ c_r)
\end{equation*}
where $\Vert \alpha \Vert_0$ represents the total number of non-zero elements of $\alpha$. 
The rank optimization problem is NP hard and so to make it more tractable, the following formulation \cite{b15} is used:
\begin{equation*}
\min_{A,B,C,\alpha} \frac{1}{2} \Vert \mathcal{X} - \sum_{r=1}^R \alpha_r ( a_r \circ b_r \circ c_r) \Vert_F^2 + \gamma \Vert \alpha \Vert_1
\end{equation*}
where $\gamma>0$ is the regularization parameter and the objective function is a composition of smooth and non-smooth functions. Our formulation includes a Tikhonov type regularization:
\begin{equation*}
    \min \frac{1}{2} \Vert \mathcal{X} - \sum_{r=1}^R \alpha_r ( a_r \circ b_r \circ c_r) \Vert_F^2  +\frac{\lambda}{2} (\Vert A\Vert_F^2+\Vert B\Vert_F^2+\Vert C\Vert_F^2)+\gamma \Vert \alpha \Vert_1.
\end{equation*}
The added Tikhonov regularization has the effect of forcing the factor matrices to have the equal norm. Moreover, this formulation and its numerical methods described later give an overall improvement in the accuracy and thus, memory requirements of the tensor model found in \cite{b15}.

\subsection{Organization}
Our paper is organized as follows. In Section 2, we provide some notations and terminologies
used throughout this paper. In Section 3,
we formulate an $l_1$-regularization optimization to the low-rank approximation of tensors.
In Section 4, we describe a numerical method to solve the $l_1$-regularization optimization
by using a proximal alternating minimization technique for the rank and an alternating least-squares for the decomposition.
In Section 5, we provide an analysis of convergence of the numerical methods.
The numerical experiments in Section 6 consist of simulated low rank tensor, color images and videos.
Finally, our conclusion and future work are given in Section 7.

\section{Preliminaries}
We denote the scalars in $\mathbb{R}$ with lower-case letters $(a,b,\ldots)$ and the vectors with lower-case letters $({a},{b},\ldots)$.  The matrices are written as upper-case letters $({A}, {B},\ldots)$ and the symbols for tensors are calligraphic letters $(\mathcal{A},\mathcal{B},\ldots)$. The subscripts represent the following scalars:  $\mathcal{(A)}_{ijk}=a_{ijk}$, $({A})_{ij}=a_{ij}$, $({a})_i=a_i$ and the $r$-th column of a matrix ${A}$ is ${a_r}$. The matrix sequence is denoted $\{{A}^k\}$. An Nth order tensor $\mathcal{X} \in \mathbb{R}^{I_1 \times I_2 \times \cdots \times I_N}$ is a multidimensional array with entries $\mathcal{(X)}_{i_1i_2\cdots i_N}=x_{i_1i_2\cdots i_N}$ for $i_k \in \{1, \hdots, I_k \}$ where $k \in {1,\hdots,N}$. In particular, a third order tensor $\mathcal{X} \in \mathbb{R}^{I \times J \times K}$ is a multidimensional array with entries $x_{i j k}$ for $i \in \{1, \hdots, I \}$, $j \in \{1, \hdots, J \}$ and $k \in \{1, \hdots, K \}$.

Here we present some standard definitions and relations in tensor analysis. The Kronecker product of two vectors $a\in \mathbb{R}^I$ and $b\in \mathbb{R}^J$ is denoted by $a \otimes b \in \mathbb{R}^{IJ}$:
\begin{equation*}
a \otimes b = \left(a_1 b^T \hdots  a_I b^T \right)^T.
\end{equation*}
The Khatri-Rao (column-wise Kronecker) product (see\cite{b12}) of two matrices $A \in \mathbb{R}^{I\times J}$ and $B \in \mathbb{R}^{K\times J}$ is defined as  
\begin{equation*}
    A \odot B = ( a_1 \otimes b_1 \hdots a_J \otimes b_J).
\end{equation*}
 The outer product of three vectors $a \in \mathbb{R}^I$, $b \in \mathbb{R}^J$, $c \in \mathbb{R}^K$ is a third order tensor $\mathcal{X}= a \circ b \circ c$ with the entries defined as follows:
\begin{equation*}
x_{ijk} = a_i b_j c_k.
\end{equation*}

\begin{definition}[vec]
Given a matrix $W \in \mathbb{R}^{I \times J}$, the function $\text{vec}: \mathbb{R}^{I \times J} \rightarrow \mathbb{R}^{I \cdot J}$ where $vec(W)=v$ is a vector of size $I\cdot J$ obtained from column-stacking the column vectors of  $W$; i.e. 
$$\text{vec}(W)_l = v(l)=w_{ij}$$
where  $l=j+(k-1)J$.
 \end{definition}

The vectorization of a third order tensor $\mathcal{X} \in \mathbb{R}^{I \times J \times K}$ is the process of transforming the tensor into a column vector, the $\text{vec} : \mathbb{R}^{I \times J \times K} \to \mathbb{R}^{IJK} $ map is defined as   
\begin{equation*}
\text{vec} (\mathcal{X})_{\beta(i,j,k)} = x_{ijk}
\end{equation*} 
where $\beta(i,j,k)=i+(j-1)I+(k-1)IJ$. Using the definitions above, we get 
\begin{equation*}
\text{vec} ( a \circ b \circ c ) = c \otimes b \otimes a
\end{equation*}

\begin{definition}[Mode-$n$ matricization]
Matricization is the process of reordering the elements of an $N$th order tensor into a matrix. The mode-$n$ matricization of a tensor $\mathcal{X} \in \mathbb{R}^{I_1 \times I_2 \times \cdots \times I_N}$ is denoted by $\mathcal{X}_{(n)}$ and arranges the mode-$n$ columns to be the columns of the resulting matrix.  The mode-$n$ column, ${x_{i_1\cdots i_{n-1} : i_{n+1} \cdots i_N}}$, is a vector obtained by fixing every index with the exception of the $n$th index.
\end{definition}

If we use a map to express such matricization process for any $N$th order tensor $\mathcal{T} \in \mathbb{R}^{I_1 \times I_2 \times \cdots \times I_N}$, that is, the tensor element $(i_1, i_2, \dots, i_N)$ maps to matrix element $(i_n, j)$, then there is a formula to calculate $j$: 
$$j=1+\sum_{\substack{k=1\\ k\neq n}}^{N}(i_k -1)J_k \quad \text{with} \quad J_k =\prod_{\substack{m=1\\ m\neq n}}^{k-1}I_m.$$

For example, the tensor unfolding or matricization of a third order tensor $\mathcal{X}$ is the process or rearranging the elements of $\mathcal{X}$ into a matrix. The mode-$n$ ($n=1,2,3$) matricization is denoted by $\mathcal{X}_{(n)}$ and the elements of it can be expressed by the following relations:
\begin{equation*}
    \mathcal{X}_{(1)} (i,l) =  x_{ijk}, \, \,  \text{where} \,\, l=j+(k-1)J \,\, \text{and} \, \, \mathcal{X}_{(1)} \in \mathbb{R}^{I \times JK}
\end{equation*}
\begin{equation*}
    \mathcal{X}_{(2)} (j,l) =  x_{ijk}, \, \,  \text{where} \,\, l=i+(k-1)I \,\, \text{and} \, \,\mathcal{X}_{(2)} \in \mathbb{R}^{J \times KI}
\end{equation*}
\begin{equation*}
    \mathcal{X}_{(3)} (k ,l) =  x_{ijk}, \, \,  \text{where} \,\, l=i+(j-1)I  \,\, \text{and} \, \, \mathcal{X}_{(3)} \in \mathbb{R}^{K \times IJ}
\end{equation*}

\subsection{CP decomposition and the Alternating Least-Squares Method}
In 1927, Hitchcock \cite{Hitch1}\cite{Hitch2} proposed the idea of the polyadic form of a tensor, i.e., expressing a tensor as the sum of a finite number of rank-one tensors. Today, this decomposition is called the canonical polyadic (CP); it is known as CANDECOMP or PARAFAC. It has been extensively applied to many problems in various engineering \cite{Sid1,Sid2,Acar,DeVos} and science \cite{Smilde,Kroonenberg}. The well-known iterative method for implementing the sum of rank one terms is the Alternating Least-Squares (ALS) technique. Independently, the ALS was introduced by Carrol and Chang \cite{JJ} and Harshman \cite{RAH} in 1970. Among those numerical algorithms, the ALS method is the most popular one since it is robust. However, the ALS has some drawbacks. For example, the convergence of ALS can be extremely slow.\\

The CP decomposition of a given third order tensor $\mathcal{X} \in \mathbb{R}^{I \times J \times K}$ factorizes it to a sum of rank one tensors.
\begin{equation}
\mathcal{X} \approx \sum_{r=1}^R \alpha_r ( a_r \circ b_r \circ c_r)
\end{equation}
For simplicity we use the notation $[A,B,C,\alpha]_R$ to represent the sum on the right hand side of the equation above, where  $A\in \mathbb{R}^{I\times R}$, $B\in \mathbb{R}^{J\times R}$ and $ C \in \mathbb{R}^{K\times R}$ are called factor matrices.
\begin{equation*}
    A=[a_1 \hdots a_R],\quad B=[b_1\hdots b_R], \quad C=[c_1 \hdots c_R]
\end{equation*}
The CP decomposition problem can be formulated as an optimization problem. Given $R$ the goal is to find vectors $a_r, b_r, c_r$, such that the distance between the tensor $\mathcal{X} $ and the sum of the outer products of $a_r, b_r, c_r$ is minimized. The Frobenius norm (sum of squares of the entries) is mainly used to measure the distance.  

\begin{equation} \label{eq2}
 \min_{A,B,C,\alpha} \frac{1}{2} \Vert \mathcal{X} - [A,B,C,\alpha]_R \Vert_F^2 
\end{equation} 
Using the Khatri-Rao product, the objective function in (\ref{eq2}) can be stated in the following four equivalent forms:
\begin{equation} \label{eq3}
    \frac{1}{2} \Vert \mathcal{X}_{(1)} - A \, \text{diag} (\alpha) (C  \odot B  )^T \Vert_F^2, 
\end{equation}
\begin{equation}\label{eq4}
     \frac{1}{2} \Vert \mathcal{X}_{(2)} - B \, \text{diag} (\alpha) (C  \odot A  )^T \Vert_F^2, 
\end{equation}
\begin{equation}\label{eq5}
 \frac{1}{2} \Vert \mathcal{X}_{(2)} - C \, \text{diag} (\alpha) (B  \odot A  )^T \Vert_F^2, 
\end{equation}
\begin{center}
and
\end{center}
\begin{equation}\label{eq6}
   \frac{1}{2} \Vert \text{vec}(\mathcal{X}) - \text{vec}([A,B,C,\alpha]_R) \Vert_2^2  
\end{equation}
All the functions in (\ref{eq3}), (\ref{eq4}), (\ref{eq5}) and (\ref{eq6}) are linear least squares problems with respect to matrices A, B, C and vector $\alpha$.  To find approximations to A,B,C, and $\alpha$, these four optimization problems (\ref{eq3})-(\ref{eq6}) are implemented iteratively and the minimizers are updated between each optimization problems (via Gauss-Seidel sweep) with a stopping criteria. This technique is called the Alternating Least Squares (ALS) Method. The ALS method is popular since it is robust and easily implementable. However, the ALS has some drawbacks. For example, the convergence of ALS can be extremely slow. Another drawback is the requirement of a tensor rank $R$ before a CP decomposition is approximated. The next sections deal with tensor rank approximation.

\section{Rank Approximation of a Tensor}
The problem of finding the rank of a tensor can be formulated as a constrained optimization problem. 
\begin{equation*}
\min_\alpha \Vert \alpha \Vert_0 \, \, \, 
\text{s.t.} \, \, \, \,  \mathcal{X}=[A,B,C, \alpha]_R
\end{equation*}
where $\Vert \alpha \Vert_0$ represents the total number of non-zero elements of $\alpha$. Since the problem is NP hard (ref), we replace $\Vert \alpha \Vert_0 $ by the $\ell_1 $ norm of $\alpha $. The $\ell_1$ norm is defined as the sum of absolute value of the elements of $\alpha$. So the rank approximation problem can be written as 
\begin{equation*}
\min_\alpha \Vert \alpha \Vert_1 \, \, \, \, \text{s.t.}  \, \, \, \mathcal{X} = [A,B,C, \alpha]_R
\end{equation*}
In order to obtain a CP decomposition of the given tensor $\mathcal{X}$ as well as the rank approximation, we formulate the rank approximation problem as follow:
\begin{equation}\label{eq71}
\min_{A,B,C,\alpha} \frac{1}{2} \Vert \mathcal{X} - [A,B,C, \alpha] \Vert_F^2 + \gamma \Vert \alpha \Vert_1
\end{equation}
where $\gamma>0$ is the regularization parameter. The objective function of the problem (\ref{eq71}) is non-convex and non-smooth. However, it is a composition of a smooth and non-smooth functions.

Moreover, it is known that CP decomposition of a tensor is unique up to scaling anf permutation of factor matrices. Note that 
\begin{equation*}
    [A,B,C,\alpha]_R=[cA,c^{-1}B,C,\alpha]_R
\end{equation*}
for a nonzero scalar $c \in \mathbb{R}$. In order to overcome the scaling indeterminacy, we add a Tikhonov type regularization term to our objective function \cite{TenDe}. Let f and g be the following:
\begin{equation}
f(A,B,C,\alpha)=\frac{1}{2} \Vert \mathcal{X} - [A,B,C, \alpha] \Vert_F^2 
\end{equation}
\begin{center}
and
\end{center}
\begin{equation}
g(\alpha)= \gamma \Vert \alpha \Vert_1
\end{equation}
which represent the fitting term and the $\ell_1 $ regularization term in (\ref{eq71}), then the rank approximation problem can be formulated as
\begin{equation}\label{eq72}
    \min f(A,B,C, \alpha) +\frac{\lambda}{2} (\Vert A\Vert_F^2)+\Vert B\Vert_F^2+\Vert C\Vert_F^2)+g(\alpha).
\end{equation}
The added Tikhonov regularization has the effect of forcing the factor matrices to have the equal norm. i.e.
\begin{equation*}
    \Vert A \Vert_F = \Vert B \Vert_F =\Vert C \Vert_F.
\end{equation*}
\cite{Paatero}, Now let $\Psi$ represent the objective function in (\ref{eq72}) collectively, then 
\begin{equation*}
    \Psi : \mathbb{R}^{R(I+J+K+1)} \to \mathbb{R}^+
\end{equation*}
where
\begin{align} \label{ObjEqu}
    \Psi (A,B,C,\alpha) &= f(A,B,C,\alpha) +\frac{\lambda}{2} ( \Vert A \Vert_F^2 + \Vert B \Vert_F^2 \\ 
    &+\Vert C \Vert_F^2 ) + g(\alpha) \nonumber
\end{align}
Let $\omega=(A, B, C, \alpha)$, when $B, C, \alpha$ are fix, we represent $f(\omega)$ by $f(A)$ and $\Psi(\omega)$ by $\Psi(A)$.

\section{Approximation of Tensor Decomposition with Tensor Rank}
In this section we propose a block coordinate descent type algorithm for solving the problem (\ref{eq72}). We consider four blocks of variables with respect to $A,B,C$ and $\alpha$. In particular, at each inner iteration, we solve the following minimization problems 
\begin{equation} \label{eq7}
A^{k+1} = \argmin_{A} \{ f(A,B^k,C^k,\alpha^k) +\frac{\lambda}{2} \Vert A \Vert_F^2\}
\end{equation}
\begin{equation} \label{eq8}
B^{k+1} = \argmin_{B} \{ f(A^{k+1},B,C^k,\alpha^k) +\frac{\lambda}{2} \Vert B \Vert_F^2\}
\end{equation}
\begin{equation} \label{eq9}
C^{k+1} = \argmin_{A} \{ f(A^{k+1},B^{k+1},C,\alpha^k) +\frac{\lambda}{2} \Vert C \Vert_F^2\}
\end{equation}
\begin{center}
and
\end{center}
\begin{equation} \label{eq10}
\alpha^{k+1} = \argmin_{\alpha} \{ L_f^{\beta^k}(A^{k+1},B^{k+1},C^{k+1},\alpha) +g(\alpha)\}
\end{equation}
where $L_f^{\beta^k}(\alpha)$ represents the proximal linearization \cite{PALM} of $f$ with respect to $\alpha$, namely
\begin{equation*}
    L_f^{\beta^k}(\alpha)=\langle \alpha-\alpha^k, \nabla f_\alpha(\alpha^k) \rangle + \frac{1}{2\beta^k} \Vert \alpha - \alpha^k \Vert^2
\end{equation*}
Note that each of the minimization problems in (\ref{eq7})-(\ref{eq10})  is strictly convex, therefore $A,B,C,\alpha$ are uniquely determined at each iteration. In fact, the subproblems in (\ref{eq7})-(\ref{eq9}) are standard liner least squares problems with an additional Tikhonov regularization term. One can see by vectorization of the objective functions, for instance, the residual term in (\ref{eq3}) can be written as follows
\begin{equation*}
    \frac{1}{2} \Vert \text{vec} (\mathcal{X}_{(1)} ) - ( ((C \odot B)  \text{diag}(\alpha)) \otimes I) \text{vec}(A) \Vert_2^2
\end{equation*}
Since the objective functions in (\ref{eq7})-(\ref{eq9}) are strictly convex, the first order optimality condition is sufficient for a point to be minimum. In other words, the exact solutions of (\ref{eq7})-(\ref{eq9}) can be given be the following normal equations
\begin{equation*}
    A (E^k (E^k)^T +\lambda I ) = \mathcal{X}_{(1)} (E^k)^T,
\end{equation*}
\begin{equation*}
    B (F^k (F^k)^T +\lambda I ) = \mathcal{X}_{(2)} (F^k)^T,
\end{equation*}
\begin{center}
and
\end{center}
\begin{equation*}
    C (G^k (G^k)^T +\lambda I ) = \mathcal{X}_{(3)} (G^k)^T
\end{equation*}
where  $E^k = \text{diag}(\alpha^k) (C^k \odot B^k)^T$, $F^k = \text{diag}(\alpha^k) (C^k \odot A^{k+1})^T$ and  $G^k = \text{diag}(\alpha^k) (B^{k+1} \odot A^{k+1})^T$.

To update $\alpha$ in (\ref{eq10}), we discuss the proximal operator first.
\begin{definition}{\textit{(proximal operator})}
Let $g:\mathbb{R}^n \to \mathbb{R}$ be a lower semicontinuous convex function, then the proximal operator of $g$ with parameter $\beta>0$ is defined as follow
\begin{equation} \label{ProxDef}
    \text{prox}_{\beta g} (y) = \argmin_x \{g(x) + \frac{1}{2\beta} \Vert x-y \Vert_2^2\}.
\end{equation}
\end{definition}
Using the proximal operator notation, the equation (\ref{eq10}) is equivalent to 
\begin{equation*} \label{prox1}
    \alpha^{k+1}= \text{prox}_{\beta^k g} (\alpha^k- \beta^k \nabla_\alpha f(\alpha^k) ).
\end{equation*}
This is easy to verify because 
\begin{align*}
    \alpha^{k+1} &= \text{prox}_{\beta g} (\alpha^k- \beta \nabla f_\alpha(\alpha^k) ) \\
    &=\argmin_\alpha \{\frac{1}{2\beta} \Vert \alpha-\alpha^k+\beta \nabla_\alpha f (\alpha^k) \Vert_2^2+g(\alpha)\} \\
    &= \argmin_\alpha \{L_f^\beta +g(\alpha)\}.
\end{align*}
\begin{remark} The proximal operator in (\ref{ProxDef}) is well-defined because the function $g(\alpha) $  is continuous and convex. Using the vec operator, we have  
\begin{align}
    \text{vec}([A,B,C,\alpha]_R) &= \sum_{r=1}^R \alpha_r \text{vec}(a_r \circ b_r \circ c_r ) \\
    &=\sum_{r=1}^R \alpha_r (c_r \otimes b_r \otimes a_r) \\
    &= M \alpha
\end{align}
where $M \in \mathbb{R}^{IJK \times R}$ is the matrix with columns $c_r \otimes b_r \otimes a_r$. Therefore we can rewrite the objective function $f$ as
\begin{equation} \label{Leastalpha}
    \frac{1}{2} \Vert \text{vec}(\mathcal{X}) - M \alpha \Vert_2^2
\end{equation}
\end{remark}  
It is easy to calculate the gradient of (\ref{Leastalpha}) with respect to $\alpha$:
\begin{equation}
    \nabla_\alpha f (A,B,C,\alpha) = M^T(M \alpha-\text{vec}(\mathcal{X}))
\end{equation}
This implies the Lipschitz continuity of the gradient of $f$ with respect to $\alpha$. The Lipschitz constant  is  $Q_\alpha =\Vert M^T M\Vert$ so we must have
\begin{align} \label{Lip}
    \Vert \nabla_\alpha f(\alpha_1)- \nabla_\alpha f(\alpha_2) \Vert \leq Q_\alpha \Vert \alpha_1 - \alpha_2 \Vert.
\end{align}

\section{Analysis of Convergence}
in this section, we study the global convergence of the proposed algorithm under mild assumptions. The Kurdyka-Lojasiewicz \cite{Kurd}, \cite{Loja}  property plays a key role in our analysis. We begin this section by stating the descent lemma.
\begin{lemma}{(Descent Lemma)}
Let $h:\mathbb{R}^n \to \mathbb{R}$ be continuously differentiable function, and $\nabla h$ is Lipschitz continuous with constand L, then for any $x, y \in \mathbb{R}^n$ we have
\begin{equation*}
    h(x)\leq h(y)+\langle x-y,\nabla h(y) \rangle +\frac{L}{2} \Vert x-y \Vert^2.
\end{equation*}
\end{lemma}
Next lemma provides the theoretical estimate for the decrease in the objective function after a single update $\alpha$.
\begin{lemma} \label{lem2}
Suppose that $\alpha^{k+1}$ is obtained by the equation (\ref{prox1}) and $0<\beta^k <1/Q_\alpha^k$, where $Q_\alpha^k$'s are defined in (\ref{Lip}), then there is a constant $N^k>0$ such that
\begin{equation}
    \Psi(\alpha^{k} ) - \Psi(\alpha^{k+1}) \geq N^k \Vert \alpha^{k+1} -\alpha^k \Vert^2.
\end{equation}
\end{lemma}
\begin{proof}
Recall that 
\begin{equation*}
    L_f^{\beta^k}(\alpha)=\langle \alpha-\alpha^k, \nabla f_\alpha(\alpha^k) \rangle + \frac{1}{2\beta^k} \Vert \alpha - \alpha^k \Vert^2,
\end{equation*}
and $\alpha^{k+1}$ is obtained by the equation 
\begin{equation*}
    \alpha^{k+1} =  \argmin_\alpha \{ L_f^{\beta^k}(\alpha) + g(\alpha) \} , 
\end{equation*}
therefore we must have 
\begin{equation} \label{lem2eq}
    L_f^{\beta^k}(\alpha^{k+1}) +g(\alpha^{k+1}) \leq g(\alpha^k).
\end{equation}
Since $\nabla_\alpha f$ is Lipschitz continuous with constant $Q_\alpha^k$, by the descent lemma we have
\begin{align*}
    f(\alpha^{k+1}) &\leq f(\alpha^k) + \langle \alpha^{k+1}-\alpha^k,\nabla_\alpha f(\alpha^k)\rangle \\
    &+\frac{ Q_\alpha^k }{2} \Vert \alpha^{k+1} - \alpha^k \Vert^2 \nonumber
\end{align*}
with (\ref{lem2eq}), the above inequality implies
\begin{align*}
    f(\alpha^{k+1})+g(\alpha^{k+1}) &\leq f(\alpha^k) +g(\alpha^k)\\ &- \left ( \frac{ 1- \beta^k  Q_\alpha^k  }{2\beta^k} \right) \Vert \alpha^{k+1} -\alpha^k \Vert^2 
\end{align*}
setting 
\begin{equation*}
    N^k=\frac{ 1- \beta  Q_\alpha^k  }{2\beta}>0
\end{equation*}
proves the lemma.
\end{proof}
\begin{remark}
Suppose that $Q_\alpha^k$'s are bounded from above by the constant $Q_\alpha$ in the previous lemma, then for fixed step-size $\beta$ where
\begin{equation*}
    0<\beta <1/Q_\alpha 
\end{equation*}
we have 
\begin{equation}
    \Psi(\alpha^{k} ) - \Psi(\alpha^{k+1}) \geq \left ( \frac{ 1- \beta  Q_\alpha  }{2\beta} \right) \Vert \alpha^{k+1} -\alpha^k \Vert^2
\end{equation}
for each $k=1,2,\hdots$ .
\end{remark}
\begin{definition} \cite{b14}
A differentiable function $h:\mathbb{R}^n \to \mathbb{R}$ is called strongly convex if there is a constant $\mu >0$ such that 
\begin{equation*}
    h(x)-h(y) \geq  \langle \nabla h(y), x-y \rangle + \frac{\mu}{2} \Vert x-y \Vert^2
\end{equation*}
for any $x,y \in \mathbb{R}^n$.
\end{definition}
\begin{lemma}\label{lem3}
Suppose that $A^{k+1}$ is obtained by equation (\ref{eq7}), then we have 
\begin{equation*}
    \Psi(A^k)- \Psi(A^{k+1})  \geq \frac{\lambda}{2}\Vert A^k -A^{k+1} \Vert_F^2
\end{equation*}
\end{lemma}
\begin{proof}
Note that the objective functions in (\ref{eq7}) is strongly convex with parameter $\lambda$ and by the first-order optimality condition we must have 
\begin{equation*}
    \nabla_A f(A^{k+1}) + \lambda A^{k+1}=0
\end{equation*}
now the strong convexity of $f+\Vert . \Vert_F^2$ yields 
\begin{align*}
    f(A^{k})+ \frac{\lambda}{2}\Vert A^k \Vert_F^2 &-f(A^{k+1})-\frac{\lambda}{2}\Vert A^{k+1}) \Vert_F^2 \\ &\geq \frac{\lambda}{2} \Vert A^k -A^{k+1} \Vert^2
\end{align*}
which implies 
\begin{equation*}
    \Psi( A^{k})-\Psi (A^{k+1})  \geq \frac{\lambda}{2} \Vert A^k -A^{k+1} \Vert^2.
\end{equation*}
This proves the lemma.
\end{proof}
\begin{remark}
Similar results hold for the blocks $B$ and $C$, if they are updated by equations (\ref{eq8}) and (\ref{eq9}). In particular, we have that 
\end{remark}

The next theorem guarantees that the value of $\Psi$ decreases monotonically at each iteration. This shows that the sequence $\{ \omega^k \}$ generated by scheme (\ref{eq7}), (\ref{eq8}), (\ref{eq9}) and (\ref{eq10}) is monotonically decreasing in value,
\begin{thm}{(Sufficient decrease property)} 
Let $\Psi$ represent the objective function in (ref) and $\omega^k = (A^{k},B^{k}, C^k, \alpha^k )$, then we have 
\begin{equation}
    \Psi(\omega^k)- \Psi(\omega^{k+1}) \geq \rho \Vert \omega^k -\omega^{k+1} \Vert^2
\end{equation}
for some positive constant $\rho$. In addition we have 
\begin{equation}
    \sum_{k=0}^\infty \Vert \omega^k -\omega^{k+1} \Vert^2  < \infty
\end{equation}
\end{thm}
\begin{proof}
By lemmas \ref{lem2} and \ref{lem3} we have 
\begin{align*}
   \Psi(\omega^k) - \Psi(\omega^{k+1}) &\geq \frac{\lambda}{2} (\Vert A^k-A^{k+1}\Vert^2 +\Vert B^k-B^{k+1} \Vert^2 \\
   &+\Vert C^k-C^{k+1} \Vert^2)+ N_\alpha \Vert \alpha^k -\alpha^{k+1}\Vert^2
\end{align*}
setting $\rho = \min \{\lambda/2, N_\alpha\}$ gives the first result. This shows that the sequence $\{\Psi(\omega^k)\}$ generated by our algorithm is decreasing. The monotonicity of $\{\Psi(\omega^k)\}$ with the fact that $\Psi$ is bounded from below, implies $\Psi(\omega^k) \to \inf \Psi =\underline{\Psi}$ as $k\to \infty$,  next let $n>2$ be a positive integer, then
\begin{align*}
    \sum_{k=0}^{n-1} \Vert \omega^k -\omega^{k+1} \Vert^2 &\leq \frac{1}{\rho} \sum_{k=0}^{n-1} \left( \Psi(\omega^k)-\Psi(\omega^{k+1})\right)\\
    &= \frac{1}{\rho} (\Psi(\omega^0)-\Psi(\omega^n))
\end{align*}
letting $n\to \infty$ proves the last statement.
\end{proof}
\begin{remark} \label{RemBound}
The sequence $\{ \omega^k \}$ generated by the scheme (\ref{eq8})-(\ref{eq10}) is bounded. The reason comes from the fact that unboundedness of $\{\omega^k \}$ occurs when at least one of the blocks $A, B, C$ or $\alpha$ gets unbounded. This never happens due to the regularization terms in the objective function $\Psi$ and the fact that $\Psi( \omega^k)$ is non-increasing.
\end{remark}
\begin{thm}
Let $\{ \omega^k \}_{k\in \mathbb{N}}$ be the sequence generated by our algorithm, then there exists a positive constant $\rho >0$ such that for any $k \in \mathbb{N}$ there is a vector $\eta^{k+1} \in \partial \Psi(\omega^{k+1})$ such that
\begin{equation*}
    \Vert \eta^{k+1} \Vert \leq \rho \Vert \omega^{k}-\omega^{k+1} \Vert
\end{equation*}
\end{thm}
\begin{proof}
Let $k$ be a positive integer. By equations (\ref{eq7}), (\ref{eq8}), (\ref{eq9}) and the first order optimality condition we have 
\begin{equation*}
    \nabla_A f(A^{k+1},B^k,C^k,\alpha^k) + \lambda A^{k+1}=0,
    \end{equation*}
    \begin{equation*}
    \nabla_B f(A^{k+1}, B^{k+1}, C^k, \alpha^k) +\lambda B^{k+1}=0 ,
\end{equation*}
and
\begin{equation*}
    \nabla_C f(A^{k+1}, B^{k+1}, C^{k+1}, \alpha^k)+\lambda C^{k+1} =0
\end{equation*}
define 
\begin{equation*}
    \eta_1^{k+1} = \nabla_A f(\omega^{k+1}) - \nabla_A f(A^{k+1}, B^{k}, C^k, \alpha^{k} )
\end{equation*}
then $\eta_1^{k+1} = \nabla_A \Psi(\omega^{k+1})$. similarly we can define vectors $\eta_2^{k+1}$, $\eta_3^{k+1}$. Next, by equation (\ref{eq10}), we have that
\begin{equation}
\alpha^{k+1} = \argmin_{\alpha} \{ L_f^{\beta^k}(A^{k+1},B^{k+1},C^{k+1},\alpha) +g(\alpha)\}
\end{equation}
hence by the optimality condition, there exists $u \in \partial g(\alpha^{k+1})$ such that 
\begin{equation*}
    \nabla_\alpha f (\alpha^{k}) + \frac{1}{\beta^k} (\alpha^{k+1}-\alpha^k) +u^{k+1} =0
\end{equation*}
define 
\begin{align*}
  \eta_4^{k+1}&=  \nabla_\alpha f(\omega^{k+1})
  -\nabla_\alpha f(\alpha^{k}) +\frac{1}{\beta^k} (\alpha^k -\alpha^{k+1})\\
  &= \nabla_\alpha f(\omega^{k+1}) +u^{k+1}
\end{align*}
so $\eta_4^{k+1} \in \partial \Psi_\alpha (\omega^{k+1})$. From these facts we have that 
\begin{equation*}
    \eta^{k+1}=(\eta_1^{k+1}, \eta_2^{k+1}, \eta_3^{k+1},  \eta_4^{k+1} ) \in \partial \Psi (\omega^{k+1})
\end{equation*}
We now estimate the norm of $\eta^{k+1}$. First note that by \ref{RemBound}, $\{\omega^k \}$ is bounded and the objective function (without the $\ell_1$ regularization term) is twice continuously differentiable, therefore as a consequence of mean value theorem, $\nabla f$ must be Lipschitz continuous. Hence there must exist a constant $P_1$ such that
\begin{align*}
    \Vert \eta_1^{k+1} \Vert &=  \Vert \nabla_A f(\omega^{k+1}) - \nabla_A f(A^{k+1}, B^{k}, C^k, \alpha^{k} ) \Vert \\
    &\leq P_1 \Vert \omega^k - \omega^{k+1} \Vert,
\end{align*}
similarly, constants $P_2$ and $P_3$ exist such that 
\begin{equation*}
   \Vert \eta_2^{k+1} \Vert\leq P_2 \Vert \omega^k - \omega^{k+1} \Vert,
\end{equation*}
and \begin{equation*}
    \Vert \eta_3^{k+1} \Vert\leq P_3 \Vert \omega^k - \omega^{k+1} \Vert.
\end{equation*}
setting $\nu = \max \{P_1, P_2, P_3, P_4\}$, gives us the result.
\end{proof}

Let $f:\mathbb{R}^n \to \mathbb{R}$ be a continuous function. The function $f$ is said to have Kurdyka-Lojasiewicz (KL) property at point $\hat{x} \in \partial f$ if there exists $\theta \in [0,1)$ such that 
\begin{equation*}
    \frac{\vert f(x)-f(\hat{x}) \vert^\theta}{\text{dist}(0,\partial f(x))}
\end{equation*}
is bounded around $\hat{x}$ \cite{b10}. A very rich class of functions satisfying the KL property is the semi-algebraic functions. These are functions where their graphs can be expressed as an algebraic set, that is 
\begin{equation*}
    \text{Graph}(f)= \bigcup_{i=1}^p \bigcap_{j=1}^q \{ x\in \mathbb{R}^n: P_{ij}=0, Q_{ij}(x) >0 \}
\end{equation*}
where $P_{ij}$'s and $Q_{ij}$'s are polynomial functions and the graph of $f$ is defined by 
\begin{equation*}
    \text{Graph}(f) =\{ (x,y) \in \mathbb{R}^{n}\times \mathbb{R} : f(x)=y \}.
\end{equation*}
Note that the univariate function  $g(x)=\vert x\vert $  is semialgebraic because 
\begin{align*}
    \text{Graph} (g) &=\{(x,y): \vert x\vert =y\}
    =\{ (x,y) : y-x=0, x>0 \} \\&\cup \{(x,y) : y+x=0, -x>0 \}
\end{align*}
The class of semi algebraic functions are closed under addition and composition \cite{b13}. Hence The objective function in (\ref{ObjEqu}) is semialgebraic therefore it satisfies KL property.
\begin{thm}
Suppose that $\{ \omega^k \}_{k \in \mathbb{N}}$ is the sequence generated by our algorithm, then $\{ \omega^k \}_{k \in \mathbb{N}}$ converges to the critical point of $\Psi$.
\end{thm}

\begin{algorithm}
\renewcommand{\algorithmicrequire}{\textbf{Input:}}
\renewcommand{\algorithmicensure} {\textbf{Output:}}
\caption{The BCD algorithm to approximate rank}
\label{alg:Framwork}
\begin{algorithmic}[1]
\REQUIRE A third order tensor $\mathcal{X}$, an upper bound $R$ of $\mbox{rank}(\mathcal{X})$, and the regularization parameters $\lambda>0$, $\gamma>0$ and the fixed stepsize $\beta$;\\
\ENSURE An approximated tensor $\mathcal{Y}$ with an estimated rank $\hat{R}$;\\
\STATE Given initial guess $\mathcal{X}^0=[A^0,B^0,C^0,\alpha^0]_R$.
\WHILE{stopping criterion not met} 
\STATE Update $A$ :\\
$E=\text{diag}(\alpha) (C \odot B)^T$ \\
$A=(\mathcal{X}_{(1)}E)/(EE^T+\lambda I)$
\STATE Update $B$ : \\
$F=\text{diag}(\alpha) (C \odot A)^T$ \\
$B=(\mathcal{X}_{(2)}F)/(FF^T+\lambda I)$
\STATE Update $C$ :\\
$G=\text{diag}(\alpha) (B \odot A)^T$ \\
$C=(\mathcal{X}_{(3)}G)/(GG^T+\lambda I)$
\STATE Update $\alpha$ :
\FOR{$r=1$ to $R$}
\STATE $M(:,r)= \text{vec}(a_r \circ b_r \circ c_r)$
\ENDFOR
\STATE $y=\alpha-\beta (M^T(M \alpha- \text{vec}(\mathcal{X})))$
\FOR{$r=1$ to $R$}
\STATE $\alpha(r)= \begin{cases} y(r)-\beta  & y(r)>\beta\\ 
0 & \vert y(r) \vert \leq \beta\\
y(r)+\beta & y(r) <-\beta
\end{cases}$
\ENDFOR
\ENDWHILE
\STATE Create tensor $\mathcal{Y}$ with factor matrices $A, B, C$ and coefficients $\alpha$.
\STATE Count the number of non-zero elements of $\alpha$ and assign it to $\hat{R}$.
\RETURN The tensor $\mathcal{Y}$ with the estimated rank $\hat{R}$.
\end{algorithmic}
\end{algorithm}


\section{Numerical Experiment and Results}
In this section we test our algorithm on tensors with different rank and dimensions. We randomly generate tensors with specified ranks and compare the performance of our algorithm with other available algorithms such as LRAT \cite{b15}. Next, we apply our algorithm on single moving object videos in order to extract the background and target object. 
\subsection{Tensor Rank Approximation}
In this subsection we test the performance of our algorithm on randomly generated cubic tensors with various dimensions and various rank. The upper bound for the rank of tensors are set to be equal to $\min\{IJ, JK, IK\}$. The results are shown in TABLE I.


\subsection{Comparison between LRAT and our algorithm}
In this subsection, we compare the performance of our proposed algorithm to LRAT \cite{b15}. We generate a random cubic tensor $\mathcal{A}\in \mathbb{R}^{5\times 5 \times 5}$ where its rank is equal to five. The comparison is based on the residual function as well as the sparsity of vector $\alpha$. The upper bound for the rank of the tested tensor is set to be equal to ten for both algorithms. 


\begin{table} \label{tbl1}

\begin{center}
\begin{tabular}{|c|c|c|c|}
\hline
&\multicolumn{3}{|c|}{\textbf{Size of Tensor}} \\
\cline{2-4} 
 & \textbf{\textit{$I, J, K=5$}}& \textbf{\textit{$I, J, K=7$}}& \textbf{\textit{$I, J, K=10$}} \\
\hline
\textbf{Actual Rank}& 5& 8 & 10 \\
\hline
\textbf{Upper bound}& 10& 15 & 20 \\
\hline
\textbf{Estimated Rank} & 5 & 8 & 12\\
\hline
\textbf{Residual error} &2.85e-1 & 1.34e-1 &1.20e-1 \\
\hline 
\textbf{Relative error} &5.17e-2 &1.05e-2 &5.00e-3 \\
\hline 
\textbf{Time} & 2.23& 3.86& 6.39\\
\hline 
\end{tabular}
\label{tab2}
\end{center}
\caption{Rank Approximation}
\end{table}
\subsection{Application in background extraction of single moving object videos}
In this subsection we apply our algorithm to extract the background of videos. 
See Figure \ref{fig:videoframes}. The video example \cite{bouwmans1,bouwmans2} is a $48 \times 48 \times 51$ with rank $23$ tensor. The relative residual error of $\Vert \mathcal{X} - \sum_r^R \alpha_r a_r \circ b_r \circ c_r \Vert_F^2 $ is $10^{-8}$.

\begin{figure}
  \includegraphics[scale=.85]{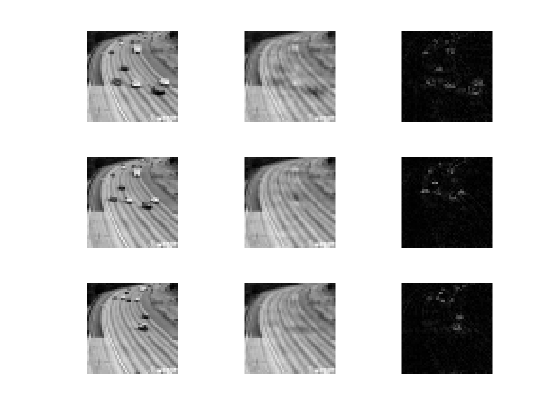}
  \caption{The original video \cite{bouwmans1,bouwmans2} is of the size $48 \times 48 \times 51$. Column 1 shows the original (11th,16th,49th) frames, column 2 shows the reconstruction (background) and column 3 shows the foreground (moving objects).}
  \label{fig:videoframes}
\end{figure}

\begin{figure}
\begin{center}
  \includegraphics[scale=.55]{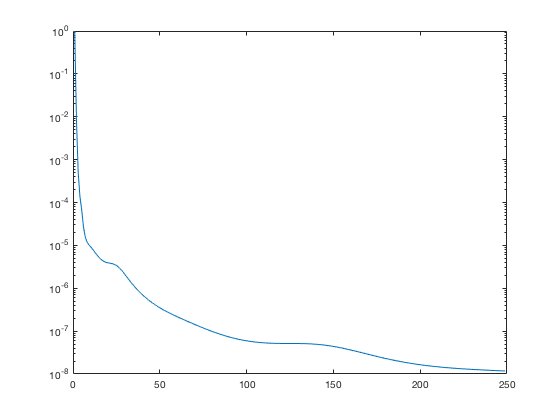}
  \caption{Residual Plot. The x-axis is the number of iterations and y-axis is the relative error term of $\Vert \mathcal{X}- \sum_{r}^{R} \alpha_r a_r \circ b_r \circ c_r  \Vert_F^2$} for the video example in Figure \ref{fig:videoframes}.
  \label{fig:residual}
  \end{center}
\end{figure}


\begin{figure}
    \centering
    \subfloat[original image ]{{\includegraphics[width=5cm]{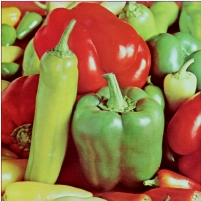} }}
    \qquad
    \subfloat[reconstructed image ]{{\includegraphics[width=5cm]{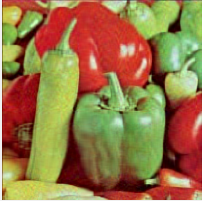} }}
    \caption{The performance of our algorithm on RGB image. The right image illustrates the compressed reconstructed version of the original image.}
    \label{fig:imeg}%
\end{figure}


\section{Conclusion}\label{sec:con}
We presented the iterative algorithm for approximating tensor rank and CP decomposition based on a sparse optimization problem. Specifically, we apply a Tikhonov regularization method for finding the decomposition and a proximal algorithm for the tensor rank. We have also provided convergence analysis and numerical experiments on color images and videos. Overall, this new tensor sparse model and its computational method dramatically improve the accuracy and memory requirements.

\end{document}